\documentclass[11pt]{article}

\usepackage[a4paper, margin=3.5cm]{geometry}
\usepackage{lmodern}
\usepackage{amsmath, amssymb, amsthm, mathtools}
\usepackage[dvipsnames]{xcolor}
\usepackage{hyperref}
\usepackage{enumitem}
\usepackage[expansion=false]{microtype}
\usepackage{booktabs}
\hypersetup{
  colorlinks=true,
  linkcolor=blue!50!black,
  citecolor=blue!50!black,
  urlcolor=blue!50!black
}

\newtheorem{theorem}{Theorem}
\newtheorem{proposition}[theorem]{Proposition}

\newtheorem{conjecture}[theorem]{Conjecture}
\theoremstyle{definition}
\newtheorem{remark}[theorem]{Remark}
\newtheorem{example}[theorem]{Example}
\newtheorem{definition}[theorem]{Definition}

\newcommand{\F}{\mathbb{F}}
\newcommand{\rs}{\mathrm{rowspace}\,}
\newcommand{\spn}{\mathrm{span}\,}
\newif\ifhighlightchanges
\highlightchangesfalse
\definecolor{RevisionOrange}{HTML}{D26900}
\newcommand{\revised}[1]{{\ifhighlightchanges\color{RevisionOrange}\hypersetup{allcolors=RevisionOrange}\fi#1}}
\newenvironment{revision}{\begingroup\ifhighlightchanges\color{RevisionOrange}\hypersetup{allcolors=RevisionOrange}\fi}{\endgroup}

\title{On set-like sunflower-free families of subspaces\\over finite fields}
\author{Kamil Otal\\
  \small T\"UB\.ITAK B\.ILGEM UEKAE\\
  \small National Research Institute of Electronics and Cryptology\\
  \small Gebze, Kocaeli, T\"urkiye\\
  \small \texttt{kamil.otal@gmail.com}}

\date{\today} 

\begin{document}
\maketitle

\begin{abstract}
The Erd\H{o}s--Rado sunflower problem admits two natural analogues in finite vector spaces, corresponding to two different ways of generalising the set-theoretic notion of a sunflower. The first, used by Ihringer and Kupavskii [FFA 110 (2026) 102746], requires the petals to be in general position over the kernel; the second, used in the subspace codes literature (cf.\ Etzion--Raviv [DAM 186 (2015) 87-97], Blokhuis--De Boeck--D'haeseleer [DCC 90 (2022) 2101-2111]), requires only that the kernel equals the pairwise intersection of distinct petals. We refer to the second version as a \emph{set-like sunflower}, following Blokhuis, De Boeck and D'haeseleer~\cite{bdh}.

In this note, we focus on the set-like setting. {We first observe by illustrations that general-position sunflower-freeness does not imply set-like sunflower-freeness.} We then present a construction of set-like $s$-sunflower-free families of $k$-spaces \revised{in $\F_q^n$ for $\lceil(n+1)/2\rceil \leq k \leq n-1$, of size $q^{n-k}$ (and $q^{n-k}+1$ for even $q$),} based on a manipulated version of the lifting construction; to our knowledge, this is the first systematic construction tailored to this setting. {Finally, we compare the size of the construction with the exact optimum for small parameters, computed by exhaustive search.}


\end{abstract}

\section{Introduction}

Let $s \geq 3$. We say that sets $S_1, \ldots, S_s$ form an \emph{$s$-sunflower} with kernel $K$ when every pairwise intersection coincides with $K$, that is, $S_i \cap S_j = K$ whenever $i \neq j$; the older literature speaks of a \emph{$\Delta(s)$-system}.\footnote{Erd\H{o}s and Rado originally wrote of $\Delta$-systems; the botanical name gained currency through the work of Deza and Frankl~\cite{deza-frankl} and through its use in Boolean circuit complexity, and has by now largely displaced the original terminology.}\footnote{Here, a family is a set, not a multiset.} Following the classical convention, the kernel is allowed to be empty: in particular, any $s$ pairwise disjoint sets form an $s$-sunflower with $K = \emptyset$. A family of $k$-sets is \emph{$s$-sunflower-free} if no $s$ of its members form an $s$-sunflower. {Erd\H{o}s and Rado~\cite{erdos-rado} showed in 1960 that, for $s \geq 3$, an $s$-sunflower-free family of $k$-sets can have at most $k!(s-1)^k$ members, and they conjectured a much sharper bound:}

\begin{conjecture}[Erd\H{o}s--Rado sunflower conjecture]
For every $s \geq 3$ there exists a constant $C$ such that every $s$-sunflower-free family of $k$-sets $\mathcal{F}$ satisfies $|\mathcal{F}| \leq (Cs)^k$.
\end{conjecture}

{Choosing one element from each of $k$ disjoint groups of size $s-1$ in every possible way produces $(s-1)^k$ sets among which no $s$ form a sunflower, so the conjectured bound would be best possible up to the constant $C$. Later on, Alweiss, Lovett, Wu, and Zhang~\cite{alwz} came within a logarithmic factor of the conjecture; together with a refinement by Bell, Chueluecha and Warnke~\cite{bcw}, the current record reads:}

\begin{proposition}[\cite{alwz,bcw}]\label{prop:alwz}
For every $s \geq 3$, any $s$-sunflower-free family of $k$-sets $\mathcal{F}$ satisfies $|\mathcal{F}| \leq (Cs \log k)^k$ for some absolute constant $C$.
\end{proposition}

\subsection{Sunflowers in finite vector spaces}

Let $q$ be a prime power, $\F_q$ the finite field with $q$ elements, and $\F_q^n$ the $n$-dimensional vector space over $\F_q$; {throughout, a \emph{$k$-space} is a $k$-dimensional subspace}. When generalising the notion of a sunflower to finite vector spaces, two natural choices arise.

\begin{definition}[General-position sunflower; Ihringer--Kupavskii~\cite{ik2026}]\label{def:gp-sunflower}
A family $S_1, \ldots, S_s$ of $k$-spaces of $\F_q^n$ is a \emph{(general-position) $s$-sunflower} with kernel $K$ if there exists a $d$-space $K$ with $K = S_i \cap S_j$ for all distinct $i, j$, and $S_1/K, \ldots, S_s/K$ are in general position, i.e.\
$\dim(S_1/K + \cdots + S_s/K) = \sum_{i=1}^s \dim(S_i/K)$.
Equivalently, $\dim(S_1 + \cdots + S_s) = d + s(k-d)$.
\end{definition}

\begin{definition}[Set-like sunflower; cf.~Etzion--Raviv~\cite{etzion-raviv}]\label{def:set-like-sunflower}
A family $S_1, \ldots, S_s$ of $k$-spaces of $\F_q^n$ is a \emph{set-like $s$-sunflower} with kernel $K$ if $K = S_i \cap S_j$ for all distinct $i, j$. As in the set case, the kernel is allowed to be the zero subspace, $K = \{0\}$.
\end{definition}

The terminology \emph{set-like sunflower} follows Blokhuis, De Boeck and D'haeseleer~\cite{bdh}. The corresponding notion, in which only the pairwise intersections are constrained, has a parallel literature in the context of subspace codes; see, for example, Etzion and Raviv~\cite{etzion-raviv}.

Every general-position sunflower is a set-like sunflower, but not conversely. The two definitions therefore lead to two distinct extremal problems. In the set case the two notions coincide, because three sets pairwise intersecting in a common set $K$ automatically satisfy the analogous general-position condition; in the vector-space case they do not, due to the failure of the inclusion-exclusion principle for sums of three or more subspaces (see Section~\ref{sec:incl-excl}).

A family of $k$-spaces is \emph{(general-position) $s$-sunflower-free} (respectively, \emph{set-like $s$-sunflower-free}) if it contains no $s$-sunflower in the corresponding sense. Since a set-like sunflower is a more permissive notion, set-like sunflower-freeness is the stronger property: every set-like $s$-sunflower-free family is also general-position $s$-sunflower-free, but not conversely.

\begin{remark}[The role of the trivial kernel]\label{rem:spreads}
Allowing the trivial kernel $K = \{0\}$ is what makes the set-like problem nontrivial. If one instead insisted on a nonzero kernel, then any family of $k$-spaces with pairwise trivial intersection --- for instance a (partial) spread --- would automatically be set-like $s$-sunflower-free, simply because no $s$ of its members share a nonzero common pairwise intersection. Whenever such families are large (for example, a spread of $\F_q^n$ into $k$-spaces when $k \mid n$, of size $\tfrac{q^n-1}{q^k-1}$), this convention trivialises the question. We therefore allow $K = \{0\}$ throughout, in line with the classical set-theoretic convention. Under this convention a (partial) spread is \emph{not} set-like sunflower-free, since any $s$ of its members form a set-like $s$-sunflower with kernel $\{0\}$; and the construction of Section~\ref{sec:construction} avoids set-like sunflowers with \emph{every} kernel, the trivial one included.
\end{remark}

\subsection{Related work and contributions}

Ihringer and Kupavskii~\cite{ik2026} studied the Erd\H{o}s--Rado problem under Definition~\ref{def:gp-sunflower}. They adapted the upper bound of Erd\H{o}s and Rado to obtain
\[
|\mathcal{F}| \;\leq\; \prod_{i=1}^k [i(s-1)]_q \;\leq\; \Bigl(\tfrac{q}{q-1}\Bigr)^{\!k} q^{(s-1)\binom{k+1}{2} - k},
\]
and, by an iterative nesting of lifted MRD codes, constructed general-position $s$-sunflower-free families approaching this bound. {Here and below, $[m]_q := \frac{q^m - 1}{q - 1} = 1 + q + \cdots + q^{m-1}$ denotes the number of $1$-spaces of $\F_q^m$.} The upper bound applies to set-like sunflower-free families as well, since the latter form a subclass of the former. \revised{It is independent of the ambient dimension $n$; in particular, for fixed $q$, $k$ and $s$, it bounds the family size uniformly over all $n \geq k$.}

In the subspace codes tradition, the focus has typically been on related notions: equidistant subspace codes (cf.~\cite{etzion-raviv}), subspace packings (cf.~\cite{ekop-subspace-packings}), and sunflower-bound type results for $k$-spaces pairwise intersecting in a point (cf.~\cite{bdh}), all in the spirit of Definition~\ref{def:set-like-sunflower}.

In this note we focus on the set-like setting. We first observe that general-position sunflower-freeness does not imply set-like sunflower-freeness: the family of Example~3.1 of Ihringer--Kupavskii~\cite{ik2026}, while correct as a general-position $s$-sunflower-free family, contains set-like $3$-sunflowers (Section~\ref{sec:incl-excl}). We then present a construction of set-like $s$-sunflower-free families of $k$-spaces of $\F_q^n$ \revised{for $\lceil(n+1)/2\rceil \leq k \leq n-1$} with size $q^{n-k}$, and $q^{n-k}+1$ for even $q$, based on a manipulated version of the lifting construction (Section~\ref{sec:construction}); to our knowledge this is the first systematic construction tailored to this setting. Finally, we compare the size of the construction with the exact optimum for small parameters, which we determine by exhaustive search and identify with classical extremal configurations in projective spaces (Section~\ref{sec:optimality}).

\section{Preliminaries}

\subsection{Maximum rank distance codes}

Let $\F_q^{k \times \ell}$ denote the set of $k \times \ell$ matrices over $\F_q$. The \emph{rank distance} on $\F_q^{k \times \ell}$ is $d_R(A, B) := \mathrm{rank}(A - B)$. A subset $\mathcal{C} \subseteq \F_q^{k \times \ell}$ with $|\mathcal{C}| \geq 2$ is a \emph{rank-metric code}; its \emph{minimum rank distance} is $d_R(\mathcal{C}) := \min\{d_R(A,B) : A, B \in \mathcal{C}, A \neq B\}$.

\begin{proposition}[Singleton-like bound~\cite{delsarte}]
For any rank-metric code $\mathcal{C} \subseteq \F_q^{k \times \ell}$,
\[
|\mathcal{C}| \;\leq\; q^{\max\{k, \ell\}(\min\{k, \ell\} - d_R(\mathcal{C}) + 1)}.
\]
\end{proposition}

A rank-metric code attaining this bound is called a \emph{maximum rank distance (MRD) code}.

One useful family of MRD codes comes from the matrix representation of $\F_{q^k}$. Let $f$ be a monic irreducible polynomial of degree $k$ over $\F_q$, and $A$ its $k \times k$ companion matrix. Then $\spn(I, A, \ldots, A^{k-1})$, the set of $\F_q$-polynomials in $A$ of degree less than $k$, is a faithful matrix representation of $\F_{q^k}$; see~\cite[\S2.5]{lidl}.

\begin{proposition}\label{prop:matrix-rep}
Let $\mathcal{C} \subseteq \F_q^{k \times k}$ be the matrix representation of $\F_{q^k}$. Then $\mathcal{C}$ is an MRD code with $d_R(\mathcal{C}) = k$.
\end{proposition}

The proof is immediate: $\mathcal{C}$ is linear, and every nonzero element is invertible.

\subsection{Subspace codes and lifting}

Let $\mathcal{P}_q(n)$ denote the set of all subspaces of $\F_q^n$. The \emph{subspace distance} is
\[
d_S(U, V) := \dim U + \dim V - 2 \dim(U \cap V).
\]
A subset $\mathcal{F} \subseteq \mathcal{P}_q(n)$ with $|\mathcal{F}| \geq 2$ is a \emph{subspace code}. Write $\mathcal{G}_q(n, k)$ for the \emph{Grassmannian}, the set of all $k$-dimensional subspaces of $\F_q^n$. If all members of $\mathcal{F}$ have the same dimension $k$, that is, $\mathcal{F} \subseteq \mathcal{G}_q(n, k)$, then $\mathcal{F}$ is a \emph{constant-dimension code}.

The \emph{lifting construction} produces large constant-dimension codes from MRD codes:

\begin{proposition}[Silva--Kschischang--K\"otter~\cite{silva}]\label{prop:lifting}
Let $\mathcal{C} \subseteq \F_q^{k \times (n-k)}$ be an MRD code with $d_R(\mathcal{C}) = \min\{k, n-k\}$, and let $I$ denote the $k \times k$ identity matrix. Then
\[
\mathcal{F}_{\mathcal{C}} \;=\; \{ \rs [I \mid A] : A \in \mathcal{C}\}
\]
is a constant-dimension subspace code of size $q^{\max\{k, n-k\}(\min\{k, n-k\} - d_R(\mathcal{C}) + 1)}$ and minimum distance $d_S(\mathcal{F}_{\mathcal{C}}) = 2 d_R(\mathcal{C})$.
\end{proposition}

\section{The role of inclusion-exclusion}\label{sec:incl-excl}

For two subspaces $U, V$ of a finite-dimensional vector space,
\[
\dim(U + V) = \dim U + \dim V - \dim(U \cap V),
\]
mirroring the inclusion-exclusion identity for two sets. The analogue \emph{fails} for three subspaces: in general,
\begin{align*}
\dim(U + V + W) \;\neq\; & \dim U + \dim V + \dim W \\
& - \dim(U \cap V) - \dim(U \cap W) - \dim(V \cap W) \\
& + \dim(U \cap V \cap W).
\end{align*}
For example, take $U = \spn\{(1,0,0), (0,1,0)\}$, $V = \spn\{(1,0,0), (0,0,1)\}$, $W = \spn\{(1,0,0), (0,1,1)\}$ in $\F_q^3$. Each pair intersects in $\spn\{(1,0,0)\}$, the triple intersection is the same line, yet $U + V + W = \F_q^3$, so $\dim(U+V+W) = 3$, whereas the right-hand side above evaluates to $2 + 2 + 2 - 1 - 1 - 1 + 1 = 4$.

This failure is precisely what separates Definitions~\ref{def:gp-sunflower} and~\ref{def:set-like-sunflower}. In a general-position $s$-sunflower the identity $\dim(S_1+\cdots+S_s) = d + s(k-d)$ holds by definition (the general-position condition is exactly this equality). In a set-like sunflower no such identity is required, and the sum dimension may be strictly less than $d + s(k-d)$.

\subsection{A set-like sunflower inside the Ihringer--Kupavskii construction}

We illustrate the gap between the two definitions by exhibiting a set-like sunflower inside the smallest explicit construction of~\cite{ik2026}.

\begin{example}[\cite{ik2026}, Example 3.1]\label{ex:ik-example}
Fix a $1$-space $T$ in $\F_q^5$, and choose $q^2 + 1$ $3$-spaces $\Pi_1, \ldots, \Pi_{q^2+1}$ through $T$ with $\Pi_i \cap \Pi_j = T$ for $i \neq j$. In $\Pi_1$ take all $q^2 + q + 1$ $2$-spaces; in each $\Pi_i$ with $i \geq 2$ take the $q^2$ $2$-spaces disjoint from $T$. The resulting family $\mathcal{F}$ has $q^4 + q^2 + q + 1$ members.
\end{example}

Ihringer and Kupavskii prove in~\cite{ik2026} that $\mathcal{F}$ is general-position $3$-sunflower-free. {The same family is not set-like $3$-sunflower-free, by the following elementary observation.}

\begin{example}\label{ex:counter}
Any two of the $q+1$ two-spaces of $\Pi_1$ that contain $T$ meet exactly in $T$, since two distinct $2$-spaces of the $3$-space $\Pi_1$ meet in a $1$-space. Hence any three of them form a set-like $3$-sunflower with kernel $T$; it is not in general position, as $\dim(S_1+S_2+S_3) = 3 < 4 = d + 3(k-d)$. Similarly, when $q \geq 3$, any three $2$-spaces of a plane $\Pi_i$, $i \geq 2$, passing through a common $1$-space $P \neq T$ and disjoint from $T$ form a set-like $3$-sunflower with kernel $P$.
\end{example}

{Analogous set-like sunflowers occur in the other constructions of~\cite[\S3.3, \S3.4]{ik2026}. Simple as it is, the observation shows that the two extremal problems are distinct, so that the set-like setting calls for constructions of its own; it does not, of course, affect the results of~\cite{ik2026}, which concern Definition~\ref{def:gp-sunflower}.}

\section{A construction for set-like sunflower-free families}\label{sec:construction}

We now give a construction of set-like $s$-sunflower-free families of $k$-spaces. The construction is inspired by the construction of almost affinely disjoint (AAD) spaces in~\cite{aad-isit, aad-2024, aad-ffa}; in particular, the simultaneous use of $A$ and $A^2$ originates from these references.

\begin{revision}
The parameters in this section satisfy $\ell=n-k\geq 1$ and
\begin{equation}\label{eq:construction-range}
n\geq 2\ell+1
\quad\Longleftrightarrow\quad
k\geq \left\lceil\frac{n+1}{2}\right\rceil
\quad\Longleftrightarrow\quad
n\leq 2k-1.
\end{equation}
Thus the construction concerns subspaces of dimension greater than half the ambient dimension. In particular, $1\leq n-k\leq k-1$, so its size $q^{n-k}$ is at most $q^{k-1}$. The formula does not give families of size $q^{n-k}$ for fixed $k$ and arbitrarily large $n$.
\end{revision}

Throughout, $W^\perp = \{x \in \F_q^n : \langle x, w \rangle = 0 \text{ for all } w \in W\}$ denotes the orthogonal complement of a subspace $W \subseteq \F_q^n$ with respect to the standard bilinear form $\langle x, y \rangle = \sum_{i=1}^n x_i y_i$; recall that $\dim W^\perp = n - \dim W$ and $(W^\perp)^\perp = W$. We use the following standard identity:

\begin{proposition}\label{prop:sum-complement}
Let $U_1, U_2, \ldots, U_s$ be subspaces of $\F_q^n$. Then
\[
(U_1 + U_2 + \cdots + U_s)^\perp \;=\; U_1^\perp \cap U_2^\perp \cap \cdots \cap U_s^\perp.
\]
\end{proposition}

For a proof see, e.g.,~\cite{otal-thesis}. This identity is a standard tool in the subspace codes literature, used for instance in~\cite{ekop-subspace-packings}.

\begin{theorem}\label{thm:construction}
\revised{Let $\ell\geq 1$, $n \geq 2\ell + 1$ and $k = n - \ell$; equivalently, let $\lceil(n+1)/2\rceil \leq k \leq n-1$ and put $\ell=n-k$.} Let $I$ denote the $\ell \times \ell$ identity matrix over $\F_q$, and let $\mathcal{C} \subseteq \F_q^{\ell \times \ell}$ be the matrix representation of $\F_{q^\ell}$ over $\F_q$, as in Proposition~\ref{prop:matrix-rep}. For a matrix $A$, write $[A]_1$ for its first column, and let $\mathbf{0}$ denote the $\ell \times (n - 2\ell - 1)$ zero matrix, appended so that each generator matrix has exactly $n$ columns (and hence $\rs [I \mid A \mid [A^2]_1 \mid \mathbf{0}]$ is an $\ell$-space of $\F_q^n$). Define
\[
\mathcal{G} \;=\; \bigl\{ \rs [I \mid A \mid [A^2]_1 \mid \mathbf{0}] : A \in \mathcal{C}\bigr\}.
\]
Then $\mathcal{F} = \{U^\perp : U \in \mathcal{G}\}$ is a set-like $s$-sunflower-free family of $k$-spaces of $\F_q^n$ of size $q^\ell = q^{n-k}$ for every $s \geq 3$.
\end{theorem}

\begin{proof}
Since $\mathcal{C}$ has $q^\ell$ elements, $|\mathcal{F}| = q^\ell$. Each $U \in \mathcal{G}$ is an $\ell$-space, so each $U^\perp$ is a $(n-\ell) = k$-space.

A set-like $3$-sunflower-free family is automatically set-like $s$-sunflower-free for $s \geq 3$ (any $s$-sunflower contains a $3$-sub-sunflower with the same kernel). It therefore suffices to show that $\mathcal{F}$ is set-like $3$-sunflower-free.

By Proposition~\ref{prop:sum-complement}, for any $U, V, W \in \mathcal{F}$ corresponding to $A, B, C \in \mathcal{C}$,
\[
\dim(U \cap V) = n - \dim(U^\perp + V^\perp), \qquad \dim(U \cap V \cap W) = n - \dim(U^\perp + V^\perp + W^\perp).
\]

\textit{Pairwise intersection.} For distinct $A, B \in \mathcal{C}$, row reduction gives
\[
\mathrm{rank}\begin{pmatrix} I & A & [A^2]_1 & \mathbf{0} \\ I & B & [B^2]_1 & \mathbf{0} \end{pmatrix}
= \mathrm{rank}\begin{pmatrix} I & A & [A^2]_1 & \mathbf{0} \\ 0 & B - A & [B^2 - A^2]_1 & \mathbf{0} \end{pmatrix}
= 2\ell,
\]
since $B - A$ is invertible (Proposition~\ref{prop:matrix-rep}). Hence $\dim(U \cap V) = n - 2\ell$.

\textit{Triple intersection.} For pairwise distinct $A, B, C \in \mathcal{C}$,
\begin{align*}
\mathrm{rank}\begin{pmatrix} I & A & [A^2]_1 & \mathbf{0} \\ I & B & [B^2]_1 & \mathbf{0} \\ I & C & [C^2]_1 & \mathbf{0} \end{pmatrix}
&\;=\; \mathrm{rank}\begin{pmatrix} I & A & [A^2]_1 & \mathbf{0} \\ 0 & B - A & [B^2 - A^2]_1 & \mathbf{0} \\ 0 & C - A & [C^2 - A^2]_1 & \mathbf{0} \end{pmatrix} \\
&\;=\; \mathrm{rank}\begin{pmatrix} I & A & [A^2]_1 & \mathbf{0} \\ 0 & I & [B + A]_1 & \mathbf{0} \\ 0 & I & [C + A]_1 & \mathbf{0} \end{pmatrix} \\
&\;=\; \mathrm{rank}\begin{pmatrix} I & A & [A^2]_1 & \mathbf{0} \\ 0 & I & [B + A]_1 & \mathbf{0} \\ 0 & 0 & [C - B]_1 & \mathbf{0} \end{pmatrix} \;=\; 2\ell + 1.
\end{align*}
Here we used that $\mathcal{C}$ is a field (so closed under multiplication, inverses, and so that $(B-A)^{-1}(B^2 - A^2) = B + A$), together with the column-extraction identities $X[Y]_1 = [XY]_1$ and $[X]_1 + [Y]_1 = [X+Y]_1$ valid for all $X, Y \in \mathcal{C}$. Since $C \neq B$ implies $C - B$ is invertible, $[C-B]_1$ is a nonzero column and contributes rank $1$. Hence $\dim(U \cap V \cap W) = n - (2\ell + 1) = n - 2\ell - 1$.

\textit{Conclusion.} For pairwise distinct $U, V, W \in \mathcal{F}$,
\[
\dim(U \cap V \cap W) \;=\; n - 2\ell - 1 \;<\; n - 2\ell \;=\; \dim(U \cap V).
\]
If $\{U, V, W\}$ were a set-like $3$-sunflower with kernel $K$, then $K = U \cap V = U \cap W = V \cap W$, which forces $K = U \cap V \cap W$ and hence $\dim(U \cap V) = \dim(U \cap V \cap W)$, contradicting the strict inequality above. Therefore $\mathcal{F}$ contains no set-like $3$-sunflower.
\end{proof}

\begin{remark}
The proof above rules out set-like $3$-sunflowers with \emph{every} kernel, including the trivial one $K = \{0\}$ (cf.\ Remark~\ref{rem:spreads}); thus $\mathcal{F}$ is set-like sunflower-free in the meaningful sense, and not merely under a nonzero-kernel convention.
\end{remark}

\begin{remark}
\revised{A small instance of Theorem~\ref{thm:construction} with $\ell>1$ has parameters} $q = 2$, $n = 5$, $\ell = 2$ (so $k = 3$), giving a set-like $s$-sunflower-free family of $4$ three-spaces in $\F_2^5$.
\end{remark}

{For even $q$, the family of Theorem~\ref{thm:construction} can be enlarged by one further element, the analogue of the point at infinity in the lifting construction.

\begin{proposition}\label{prop:infinity}
Let $q$ be even, and keep the notation \revised{and parameter range} of Theorem~\ref{thm:construction}. Let
\[
U_\infty \;=\; \{0\}^{\ell} \times \F_q^{\ell} \times \{0\}^{n-2\ell} \;=\; \rs [\,\mathbf{0}_{\ell \times \ell} \mid I \mid \mathbf{0}_{\ell \times 1} \mid \mathbf{0}\,].
\]
Then $\mathcal{F} \cup \{U_\infty^\perp\}$ is a set-like $s$-sunflower-free family of $k$-spaces of $\F_q^n$ of size $q^{n-k} + 1$ for every $s \geq 3$.
\end{proposition}

\begin{proof}
First, $U_\infty \notin \mathcal{G}$, since every member of $\mathcal{G}$ projects onto the first $\ell$ coordinates; hence the family has $q^\ell + 1$ elements. As in the proof of Theorem~\ref{thm:construction}, it suffices to show that all pairwise sums of members of $\mathcal{G} \cup \{U_\infty\}$ have dimension $2\ell$ and all triple sums dimension $2\ell + 1$. Pairs and triples inside $\mathcal{G}$ were treated there. For any $A \in \mathcal{C}$ we have $\dim(U_A + U_\infty) = 2\ell$, since the two identity blocks occupy disjoint column blocks. For distinct $A, B \in \mathcal{C}$, reducing the middle block of rows first by the top and then by the bottom one gives
\[
\mathrm{rank}
\begin{pmatrix}
I & A & [A^2]_1 & \mathbf{0} \\
I & B & [B^2]_1 & \mathbf{0} \\
\mathbf{0} & I & \mathbf{0} & \mathbf{0}
\end{pmatrix}
=
\mathrm{rank}
\begin{pmatrix}
I & A & [A^2]_1 & \mathbf{0} \\
\mathbf{0} & \mathbf{0} & [B^2 - A^2]_1 & \mathbf{0} \\
\mathbf{0} & I & \mathbf{0} & \mathbf{0}
\end{pmatrix}.
\]
Since $q$ is even, $B^2 - A^2 = (B - A)^2$ is invertible for $B \neq A$, so $[B^2 - A^2]_1 \neq 0$ and the rank is $2\ell + 1$. The conclusion follows exactly as in Theorem~\ref{thm:construction}: pairwise intersections in $\mathcal{F} \cup \{U_\infty^\perp\}$ have dimension $n - 2\ell$ and triple intersections dimension $n - 2\ell - 1$, so no set-like $3$-sunflower, and hence no set-like $s$-sunflower, occurs.
\end{proof}

\begin{remark}\label{rem:odd-q}
For odd $q$ the same extension fails. If $A \neq 0$ and $B = -A$, then $B^2 - A^2 = 0$, and the computation above gives $\dim(U_A + U_{-A} + U_\infty) = 2\ell$, equal to the three pairwise sums. Dually, $U_A^\perp \cap U_{-A}^\perp \cap U_\infty^\perp$ has the same dimension $n - 2\ell$ as the three pairwise intersections, which therefore all coincide with it; thus $\{U_A^\perp, U_{-A}^\perp, U_\infty^\perp\}$ is a set-like $3$-sunflower for every nonzero $A \in \mathcal{C}$. Whether some $\ell$-space other than $U_\infty$ can play this role for odd $q$ is not settled by this argument. We note the parallel with Section~\ref{sec:optimality}, where an extra point beyond the oval likewise exists only for even $q$.
\end{remark}}

\subsection{Comparison with the optimum for small parameters}\label{sec:optimality}

{We carry out the comparison for $s = 3$, the smallest and most restrictive case, which is also the focus of this paper. Note that a set-like $3$-sunflower-free family is set-like $s$-sunflower-free for every $s \geq 3$, while the converse fails in general: a family with no set-like $4$-sunflowers may well contain set-like $3$-sunflowers, so the extremal sizes for different values of $s$ need not coincide. Let $M_q(n,k)$ denote the maximum size of a set-like $3$-sunflower-free family of $k$-spaces of $\F_q^n$; equivalently, $M_q(n,k)$ is the size of the largest subfamily of $\mathcal{G}_q(n,k)$ containing no triple $\{U,V,W\}$ with $U \cap V = U \cap W = V \cap W$. In particular, $M_q(n,k)$ is a lower bound for the corresponding extremal size for every $s \geq 3$. We computed $M_q(n,k)$ by exhaustive search\footnote{The search enumerates the Grassmannian and solves the resulting maximum-independent-set problem in the $3$-uniform hypergraph of set-like $3$-sunflowers with an exact constraint solver. All values in Table~\ref{tab:small} except $M_4(4,3)$ were certified optimal by the solver; for $M_4(4,3)$ the solver produced a family of size $17$, and its optimality follows from the cap bound quoted below.} in the \revised{fixed-codimension} regime $k = n - 1$ (so $\ell = 1$\revised{, and $k$ varies with $n$}), where Theorem~\ref{thm:construction} produces families of size $q^{n-k} = q$, and Proposition~\ref{prop:infinity} of size $q + 1$ for even $q$, for dimensions $n = 3, 4$ and all prime powers $q \leq 4$, together with the case $(q,n) = (2,5)$.}

\begin{table}[ht]
\centering
\begin{tabular}{cccccc}
\toprule
$q$ & $n$ & $k$ & $|\mathcal{G}_q(n,k)|$ & {construction} & optimum $M_q(n,k)$ \\
\midrule
$2$ & $3$ & $2$ & $7$  & {$3$} & $4$ \\
$3$ & $3$ & $2$ & $13$ & $3$ & $4$ \\
$4$ & $3$ & $2$ & $21$ & {$5$} & $6$ \\
{$2$} &  {$4$} &  {$3$} &  {$15$} &  {$3$} &  {$8$} \\
 {$3$} &  {$4$} &  {$3$} &  {$40$} &  {$3$} &  {$10$} \\
 {$4$} &  {$4$} &  {$3$} &  {$85$} &  {$5$} &  {$17$} \\
 {$2$} &  {$5$} &  {$4$} &  {$31$} &  {$3$} &  {$16$} \\
\bottomrule
\end{tabular}
\caption{ {The best construction size in this paper (Theorem~\ref{thm:construction}, improved to $q^{n-k}+1$ by Proposition~\ref{prop:infinity} for even $q$) against the optimum $M_q(n,k)$ for $s = 3$ and $k = n-1$, obtained by exhaustive search. Each optimum equals the maximum size of a cap in $\mathrm{PG}(n-1,q)$.}}
\label{tab:small}
\end{table}

The optima in Table~\ref{tab:small} can be identified with classical objects of projective geometry. For $k = n - 1$, the members of the family are hyperplanes of $\mathrm{PG}(n-1,q)$, and any two distinct hyperplanes meet in a codimension-$2$ space; a set-like $3$-sunflower is thus a triple of hyperplanes through a common codimension-$2$ space. By duality, set-like $3$-sunflower-free families of hyperplanes correspond to point sets no three of which are collinear, that is, to \emph{caps} of $\mathrm{PG}(n-1,q)$, and $M_q(n,n-1)$ equals the maximum size of a cap in $\mathrm{PG}(n-1,q)$. These maxima are known (see, e.g., the survey of Hirschfeld and Storme~\cite{hs-packing}): in the plane $\mathrm{PG}(2,q)$ a cap is an arc, of maximum size $q+1$ for $q$ odd (an oval, by Segre's theorem) and $q+2$ for $q$ even (a hyperoval); in $\mathrm{PG}(3,q)$ with $q > 2$ the maximum is $q^2 + 1$, attained by ovoids (Bose for $q$ odd, Qvist for $q$ even); and in $\mathrm{PG}(n-1,2)$ the maximum is $2^{n-1}$, attained by the complement of a hyperplane. This accounts for every entry of the table.

The comparison is then as follows. In dimension $n = 3$ the construction has the correct order of magnitude and misses the optimum by exactly one element: $q$ against $q + 1$ for $q$ odd, and $q + 1$ (Proposition~\ref{prop:infinity}) against $q + 2$ for $q$ even. Already in dimension $n = 4$ \revised{and for $q>2$} the gap becomes multiplicative: the optimum $q^2 + 1$ grows quadratically in $q$ while the construction stays linear. For $q = 2$ the contrast is sharpest as $n$ grows \revised{with $k=n-1$}: the optimum $2^{n-1}$ is exponential in $n$ while the construction remains of constant size. \revised{This is growth at fixed codimension, not at fixed subspace dimension $k$.} Beyond $\ell = 1$ the trend persists; for instance, for $(q,n,k) = (2,5,3)$, where \revised{Theorem~\ref{thm:construction} gives $4$ members and Proposition~\ref{prop:infinity} gives $5$}, exhaustive search already exhibits set-like $3$-sunflower-free families of size at least $17$. The size $q^{n-k}$ should therefore be regarded as a first benchmark rather than as the correct order of magnitude, which motivates the first open problem below.

\section{Concluding remarks and open problems}\label{sec:concluding}

\begin{revision}
For $s\geq 3$, let $M_{q,s}(n,k)$ be the maximum size of a set-like $s$-sunflower-free family in $\mathcal{G}_q(n,k)$; thus $M_{q,3}(n,k)=M_q(n,k)$ in Section~\ref{sec:optimality}. Since every set-like $s$-sunflower-free family is also general-position $s$-sunflower-free, the bound of~\cite{ik2026} gives, for every $n\geq k$,
\begin{equation}\label{eq:uniform-bound}
M_{q,s}(n,k)\leq B_{q,s}(k):=\prod_{i=1}^k[i(s-1)]_q
\leq \left(\frac{q}{q-1}\right)^k q^{(s-1)\binom{k+1}{2}-k}.
\end{equation}
This bound is uniform in $n$. The lower bound $q^{n-k}$ from Theorem~\ref{thm:construction}, improved to $q^{n-k}+1$ for even $q$ by Proposition~\ref{prop:infinity}, is asserted only when $\lceil(n+1)/2\rceil\leq k\leq n-1$. In that range $n-k\leq k-1$, so these construction sizes are consistent with~\eqref{eq:uniform-bound}.

To separate the asymptotic questions, keep $q$ and $s$ fixed unless stated otherwise, and distinguish the following regimes.

\begin{description}[leftmargin=0pt, labelwidth=0pt, labelsep=0pt, style=nextline, itemsep=3pt]
\item[Fixed dimension $k$, with $n\to\infty$.]
The quantity $M_{q,s}(n,k)$ is bounded independently of $n$ by~\eqref{eq:uniform-bound}. It is nondecreasing in $n$, since an embedding of ambient vector spaces preserves intersections. Hence it eventually stabilises, being a bounded integer-valued sequence. The construction formula $q^{n-k}$ does not apply throughout this regime, because it requires $n\leq 2k-1$.

\item[Growing dimension $k=\Omega(n)$.]
Theorem~\ref{thm:construction} applies to the subrange $\lceil(n+1)/2\rceil\leq k\leq n-1$. Its size is $q^{\Theta(k)}$ when, in addition, $n-k=\Theta(k)$; for example, $n=2k-1$ gives $q^{k-1}$. The condition $k=\Omega(n)$ alone does not ensure a linear exponent, since $n-k$ may be much smaller than $k$. The available upper bound is $q^{O_s(k^2)}$.

\item[Fixed codimension $\ell=n-k\geq 1$, with $n\to\infty$.]
Here $k=n-\ell$ grows with $n$. This is a special case of $k=\Omega(n)$, but the construction has constant size $q^\ell$ (or $q^\ell+1$ for even $q$). The optimum need not remain bounded: for $q=2$, $s=3$ and $\ell=1$, the cap interpretation in Section~\ref{sec:optimality} gives $M_{2,3}(n,n-1)=2^{n-1}=2^k$. Thus the hyperplane examples in Table~\ref{tab:small} concern fixed codimension, not fixed $k$.
\end{description}

Several natural questions remain:

\begin{enumerate}[leftmargin=*, itemsep=3pt]
  \item For fixed $q,s$ and $k,n-k\to\infty$ with $\lceil(n+1)/2\rceil\leq k\leq n-1$, what is the order of magnitude of $M_{q,s}(n,k)$? How large is the gap from the corresponding general-position extremal size when both are evaluated at the same parameters?
  \item Can the uniform upper bound~\eqref{eq:uniform-bound} be improved to $M_{q,s}(n,k)\leq q^{C_s k}$ for some constant $C_s$ depending only on $s$, for all prime powers $q$ and all $n\geq k\geq 2$? No assumption $k=\Omega(n)$ is imposed on this upper-bound question. Directly applying Proposition~\ref{prop:alwz} to the sets of $[k]_q$ one-spaces gives only $(Cs\log[k]_q)^{[k]_q}$, which is much weaker; for $q=2$ and $s=3$ it is roughly $(Ck)^{2^k}$ rather than $2^{O(k^2)}$. A vector-space version of the methods of~\cite{alwz,bcw} might improve this. In the balanced range $n=2k-1$, Theorem~\ref{thm:construction} already gives $q^{k-1}$, so a linear exponent in $k$ would have the correct order there. Already for $q=2$ and $s=3$, is $M_{2,3}(n,k)\leq 2^{Ck}$ for an absolute constant $C$, uniformly in $n$?
  \item Within the admissible range~\eqref{eq:construction-range}, can Theorem~\ref{thm:construction} be iterated or nested, in the spirit of~\cite{ik2026}, while preserving set-like sunflower-freeness? What sizes and parameter ranges would such a construction achieve?
  \item For fixed $q,k,s$, what is the eventual value of $M_{q,s}(n,k)$, and at what ambient dimension does it stabilise? Separately, for fixed $q,s,\ell$ with $k=n-\ell$, how does $M_{q,s}(n,n-\ell)$ grow with $n$? The cap examples concern the second question.
  \item Is there a natural intermediate notion (for instance, requiring general position only up to triples, or only on the level of $r$-fold sums for fixed $r$) that interpolates between the two settings?
\end{enumerate}
\end{revision}



\end{document}